\documentclass[12pt,reqno]{article}

\usepackage[usenames]{color}
\usepackage{amssymb}
\usepackage{amsmath}
\usepackage{amsthm}
\usepackage{amsfonts}
\usepackage{amscd}
\usepackage{graphicx}
\usepackage{tikz}

\usepackage[colorlinks=true,
linkcolor=webgreen,
filecolor=webbrown,
citecolor=webgreen]{hyperref}

\definecolor{webgreen}{rgb}{0,.5,0}
\definecolor{webbrown}{rgb}{.6,0,0}

\usepackage{color}
\usepackage{fullpage}
\usepackage{float}

\usepackage{graphics}
\usepackage{latexsym}
\usepackage{epsf}

\newcommand{\seqnum}[1]{\href{https://oeis.org/#1}{\rm \underline{#1}}}

\newcommand{\median}{\textnormal{median}}

\begin{document}


\theoremstyle{plain}
\newtheorem{theorem}{Theorem}
\newtheorem{corollary}[theorem]{Corollary}
\newtheorem{lemma}[theorem]{Lemma}
\newtheorem{proposition}[theorem]{Proposition}

\theoremstyle{definition}
\newtheorem{definition}[theorem]{Definition}
\newtheorem{example}[theorem]{Example}
\newtheorem{conjecture}[theorem]{Conjecture}

\theoremstyle{remark}
\newtheorem{remark}[theorem]{Remark}

\begin{center}
\vskip 1cm{\LARGE\bf Proofs of some Conjectures from the OEIS
}
\vskip 1cm
\large
Sela Fried \\
Department of Computer Science \\
Israel Academic College in Ramat Gan\\
Pinhas Rutenberg 87, Ramat Gan \\
52275 Israel \\
\href{mailto:friedsela@gmail.com}{\tt friedsela@gmail.com} \\
\end{center}

\vskip .2 in

\begin{abstract}
In this work we resolve several conjectures stated in the On-Line Encyclopedia of Integer sequences.
\end{abstract}
	
\section{Introduction}
The On-Line Encyclopedia of Integer sequences (OEIS) \cite{SL} contains over $375,000$ sequences and many thousands of conjectures about them. The purpose of this work is to resolve some of these conjectures. The methods we apply are mainly elementary, except, perhaps, our use of generating functions and the Euler-McLaurin formula.

\section{Main results}

\subsection{The number of bits required to represent \texorpdfstring{$\binom{2^n}{2^{n-1}}$}{}}

The statement of the following theorem was conjectured in \seqnum{A112884}. For its proof and also for proof of the conjecture in the next section, we shall need a useful tool from numerical analysis, namely the Euler-McLaurin formula. The reader may find the formula in its most general form, for example, in \cite[298 on p.\ 524]{K}. Nevertheless, for our needs, the most simple form of the formula (e.g., \cite[296 on p.\ 521]{K}) suffices: Let $\ell<m$ be two real numbers and let $f$ be a differentiable function on the interval $[\ell,m]$. Then \[
\sum_{k=\ell}^m f(k) =   \int_{\ell}^{m} f(x)dx+\frac{f(\ell)+f(m)}{2}+ \int_{\ell}^{m}f'(x)P_{1}(x)dx,\] where $P_1(x) = x - \lfloor x\rfloor -1/2$. Suppose that $f'$ is positive and monotonically decreasing in $[\ell, m]$. By the second mean-value theorem (e.g., \cite[Corollary 4.2 on p.\ 138]{W}), for some $c\in[\ell, m]$ we have \[\int_\ell^mf'(x)P_{1}(x)dx=f'(\ell)\int_{\ell}^cP_{1}(x)dx.\] Since $P_1$ is periodic with period $1$, mean-value zero, and $|P_1(x)\leq 1/2|$ for every $x$ (see \cite{BW}), we have
\[-\frac{1}{8}\leq \int_{\ell}^cP_{1}(x)\leq 0.\]
Thus,
\begin{align}
\int_{\ell}^{m} f(x)dx+\frac{f(\ell)+f(m)}{2} - \frac{f'(\ell)}{8}&\leq \sum_{k=\ell}^m f(k)\nonumber\\
&\leq \int_{\ell}^{m} f(x)dx+\frac{f(\ell)+f(m)}{2}. \label{012}
\end{align}

\begin{theorem}
Let $n\geq 1$ be an integer. The number of binary bits required to represent $\binom{2^n}{2^{n-1}}$ is $2^n -\lfloor n/2\rfloor$.
\end{theorem}

\begin{proof}
The number of binary bits required to represent a nonnegative integer $m$ is $\lfloor \log_2 m\rfloor + 1$. We have 
\begin{equation}\label{p1}
\log_2 \binom{2^{n}}{2^{n-1}}=\sum_{k=2^{n-1}+1}^{2^{n}}\log_{2}k-\sum_{k=1}^{2^{n-1}}\log_{2}k.
\end{equation}
Applying \eqref{012} on the function $f(x)=\log_2x$, we obtain 
\begin{align}
\left[x\log_{2}x-\frac{x}{\ln2}\right]_{\ell}^{m}+\frac{\log_{2}\ell+\log_{2}m}{2}-\frac{1}{8\ell\ln2}&\leq\sum_{k=\ell}^{m}\log_{2}k\nonumber\\&\leq\left[x\log_{2}x-\frac{x}{\ln2}\right]_{\ell}^{m}+\frac{\log_{2}\ell+\log_{2}m}{2}. \nonumber   
\end{align} Applying these inequalities on \eqref{p1}, we obtain 
\begin{align}
&2^{n}-\frac{n}{2}+1-2^{n-1}\left(1+\frac{1}{2\cdot 2^{n-1}}\right)\log_{2}\left(1+\frac{1}{2^{n-1}}\right)-\frac{1}{8(2^{n-1}+1)\ln2}\nonumber\\
&\leq\log_{2}\binom{2^{n}}{2^{n-1}}\nonumber\\
&\leq 2^{n}-\frac{n}{2}+1-2^{n-1}\left(1+\frac{1}{2\cdot 2^{n-1}}\right)\log_{2}\left(1+\frac{1}{2^{n-1}}\right)+\frac{1}{8\ln2}.    \nonumber
\end{align} 
Using the well known inequalities \[\frac{x}{1+x}\leq \ln(1+x)\leq x,\] which hold for every $x>-1$, we have
\[
\frac{1}{\ln2}\left(1+\frac{1}{2^{n}}\right)\left(1-\frac{1}{2^{n-1}+1}\right)
\leq 2^{n-1}\left(1+\frac{1}{2\cdot2^{n-1}}\right)\log_{2}\left(1+\frac{1}{2^{n-1}}\right)
\leq\frac{1}{\ln2}\left(1+\frac{1}{2^{n}}\right). \]
Hence, for $n\geq 5$ (actually, for $n\geq 2$), we have
\[
2^{n}-\frac{n}{2}-\frac{499}{1000}\leq\log_{2}\binom{2^{n}}{2^{n-1}} \leq 2^{n}-\frac{n}{2} - \frac{1}{50}.\]
Thus, \[\left\lfloor \log_{2}\binom{2^{n}}{2^{n-1}}\right\rfloor +1 = 2^n -\left\lfloor \frac{n}{2}\right\rfloor. \qedhere
\] 
\end{proof}

\subsection{Floor of sum of the first \texorpdfstring{$10^n$}{} cube roots}

The statement of the following theorem was conjectured in \seqnum{A136269}.

\begin{theorem}
Let $n\geq 1$ be an integer, which is divisible by $3$. Then \[ \left\lfloor\sum_{i=1}^{10^n}\sqrt[3]{i}\right\rfloor=\frac{3\cdot10^{4n/3}}{4}+5\cdot10^{n/3-1}-1.\]
\end{theorem}

\begin{proof}
Let $m$ be a positive integer. Applying \eqref{012} on the function $f(x)=\sqrt[3]{x}$, we obtain 

\[\frac{3(\sqrt[3]{m^4}-1)}{4}+\frac{\sqrt[3]{m}+1}{2}-\frac{1}{24}\leq \sum_{i=1}^m\sqrt[3]{i}\leq \frac{3(\sqrt[3]{m^4}-1)}{4}+\frac{\sqrt[3]{m}+1}{2}.\]
Now, taking $m=10^n$, we have
\[
\frac{3(\sqrt[3]{10^{4n}}-1)}{4}+\frac{\sqrt[3]{10^n}+1}{2}=\frac{3(10^{4n/3}-1)}{4}+\frac{10^{n/3}+1}{2}    =\frac{3\cdot10^{4n/3}}{4}+5\cdot10^{n/3-1}-\frac{1}{4}.\]
Since $10^{4n/3}$ is divisible by $4$, we conclude that \[\left\lfloor\sum_{i=1}^{10^n}\sqrt[3]{i}\right\rfloor=\frac{3\cdot10^{4n/3}}{4}+5\cdot10^{n/3-1}-1. \qedhere\] 
\end{proof}

\subsection{Median absolute deviation of \texorpdfstring{$\{2k^2\;:\; k=1,\ldots, n\}$}{}}

The statement of the following theorem was conjectured in
\seqnum{A345318}. The sequence deals with a statistical measure of spread called the \emph{median absolute deviation} (e.g., \cite[p.~291]{H}). 
Notice that the purpose of the factor $2$ in $2k^2$ is merely to ensure that the resulting sequence is an integer sequence. We omit this factor in our analysis.

\begin{theorem}
Let $n\geq 1$ be an integer and define $A_n = \{k^2 \;:\;k=1,2,\ldots,n\}$. Let $a_n$ denote the median absolute deviation of $A_n$, i.e., \[a_n = \median(\{|x - \median(A_n)| \;:\; x\in A_n\}).\]  Then $\lim_{n\to\infty}a_n/n^2 = \sqrt{3}/8$. 
\end{theorem}

\begin{proof}
Assume that $n$ is odd (the case of even $n$ is similar and we omit the details). Then, $\median(A_n) = (n+1)^2/4$. We shall make use of the well-known fact that every square is the sum of odd numbers. More precisely, for every integer $m\geq 1$, we have 
\[\sum_{k = 0}^{m-1}(2k-1) = m^2.\]  Let  $B_n =\{|x - \median(A_n)| \;:\; x\in A_n\}$. Then
\begin{align}
B_n& = \left\{ \left|\frac{(n+1)^{2}}{4}\right|-k^{2}\;:\;k=1,2,\ldots,n\right\} \nonumber    \\
&=\left\{ \frac{(n+1)^{2}}{4}-k^{2}:\;k=1,\ldots,\frac{n+1}{2}\right\} \bigcup\left\{ k^{2}-\frac{(n+1)^{2}}{4}:\;k=\frac{n+1}{2}+1,\ldots,n\right\} \nonumber\\
&=\left\{ \sum_{j=k}^{\frac{n+1}{2}-1}(2j+1):\;k=1,\ldots,\frac{n+1}{2}\right\} \bigcup\left\{ \sum_{j=\frac{n+1}{2}}^{k-1}(2j+1):\;k=\frac{n+1}{2}+1,\ldots,n\right\} \nonumber\\
&=\left\{ (k-1)n-(k-2)(k-1):\;k=1,\ldots,\frac{n+1}{2}\right\} \bigcup\left\{ kn+k(k+1):\;k=1,\ldots,\frac{n-1}{2}\right\}. \nonumber
\end{align}
Let us denote by $L_n$ (resp.\ $R_n$) the set on the left (resp.\ right) of the last union symbol. Clearly, $kn+k(k+1)$ is monotonically increasing with $k$ and it is easy to see that $(k-1)n-(k-2)(k-1)$ is monotonically increasing with $k$, for $k=1,2,\ldots,(n+1)/2$. We distinguish between two cases:
\begin{enumerate}
\item $a_n \in L_n$: In this case there are $k_L \in\{1,2,\ldots,(n+1)/2\}$ and $k_R\in \{1,2,\ldots,(n-1)/2\}$ such that 
\begin{align}
&k_L - 1 + k_R = \frac{n-1}{2}, \label{1}\\
&(k_L-1)n-(k_L-2)(k_L-1)\geq k_Rn+k_R(k_R+1),\label{2}\\
&(k_L-1)n-(k_L-2)(k_L-1)< (k_R+1)n+(k_R+1)(k_R+2).\label{3}
\end{align} Solving \eqref{1} for $k_R$ and substituting it into \eqref{2} and \eqref{3} we obtain 
\begin{align}
2k_{L}^{2}-(3n+5)k_{L}+\frac{3n^{2}+10n+11}{4}\leq 0,\nonumber\\
2k_{L}^{2}-(3n+7)k_{L}+\frac{3n^{2}+18n+23}{4}>0.\nonumber
\end{align} Solving these two inequalities we conclude that
\[k_L\in
\left(\frac{3n+5-\sqrt{3n^{2}+10n+3}}{4},\frac{3n+7-\sqrt{3}(n+1)}{4}\right].\] Since $a_n =(k_L-1)n-(k_L-2)(k_L-1)$, we conclude that
\begin{align}\frac{(n-1)\sqrt{3n^{2}+10n+3}}{8}\leq a_n \leq \frac{(n + 1)^2\sqrt{3}}{8}.\label{10}\end{align} 

\item $a_n \in R_n$: In this case there are $k_L \in\{1,2,\ldots,(n+1)/2\}$ and $k_R\in \{1,2,\ldots,(n-1)/2\}$ such that 
\begin{align}
&k_R - 1 + k_L = \frac{n-1}{2}, \label{100}\\
&k_Rn+k_R(k_R+1)\geq(k_L-1)n-(k_L-2)(k_L-1),\label{200}\\
&k_Rn+k_R(k_R+1)<k_Ln-(k_L-1)k_L.\label{300}
\end{align} Solving \eqref{100} for $k_L$ and substituting it into \eqref{200} and \eqref{300} we obtain 
\begin{align}
2k_{R}^{2}+(n+3)k_{R}-\frac{n^{2}+2n-3}{4}\geq 0,\nonumber\\
2k_{R}^{2}+(n+1)k_{R}-\frac{(n+1)^{2}}{4}<0.\nonumber
\end{align} Solving these two inequalities we conclude that
\[k_{R}\in\left[\frac{-n-3+\sqrt{3n^{2}+10n+3}}{4},\frac{(\sqrt{3}-1)(n+1)}{4}\right).\] Since $a_n =k_{R}n+k_{R}(k_{R}+1)$, we conclude that
\begin{align}\frac{(n-1)\sqrt{3n^{2}+10n+3}}{8}\leq a_n \leq \frac{(n + 1)^2\sqrt{3}}{8}.\label{800}\end{align} 
\end{enumerate}
Dividing \eqref{10} and \eqref{800} by $n^2$ and letting $n$ go to infinity, the assertion follows.
\end{proof}

\subsection{The integer lattice and its layers}

\noindent Let $k$ and $n$ be two nonnegative integers and let $m$ be a natural number. Denote by $a_{n,k}^{(m)}$ the number of integer points $(x_1,\ldots,x_m)$, such that $\max_{1\leq i\leq m} |x_i|\leq k$ and $\sum_{i=1}^m|x_i|\leq n$. Sequence \seqnum{A371835} is concerned with the array $\left(a_{n,k}^{(3)}\right)_{\substack{n\geq 0 \\0\leq k\leq n}}$.

Denote by $b_{n,k}^{(m)}$ the number of integer points $(x_1,\ldots,x_m)$, such that $\max_{1\leq i\leq m} |x_i|\leq k$ and $\sum_{i=1}^m|x_i|= n$. 
Fixing $k$ and $m$, we denote by $B_{k}^{(m)}(x)$ the generating function for the numbers $b_{n,k}^{(m)}$. In the following theorem we calculate $B_{k}^{(m)}(x)$.  We then use this result to resolve and extend a conjecture stated in \seqnum{A371835}, regarding a closed formula for $b_{n,k}^{(3)}$. 

This problem falls under the field of additive combinatorics (e.g., \cite[Section 2.5]{B}). 

\begin{theorem} 
We have
\begin{equation}\label{1gw}
B_{k}^{(m)}(x)=\left(1+2\sum_{i=1}^kx^i\right)^m.
\end{equation}
\end{theorem}

\begin{proof}
First, notice that \[b_{n,0}^{(m)}=
\begin{cases}
1,& \textnormal{if } n=0;\\
0,& \textnormal{otherwise}.\\
\end{cases}\] Thus $B_0^{(m)} = 1$. Assuming that $k\geq 1$, we proceed by induction on $m\geq 1$. For $m=1$ we have \[b_{n,k}^{(1)} = 
\begin{cases}
1,&\text{if $n=0$};\\
2,&\text{if $1\leq n\leq k$};\\
0,&\text{otherwise}.
\end{cases}\] Thus, \eqref{1gw} holds in this case. Assume now that \eqref{1gw} holds for $m\geq 1$. We have the recursive relation \[b_{n,k}^{(m+1)}=\sum_{i=0}^{m+1}2^{i}\binom{m+1}{i}b_{n-ik,k-1}^{(m+1-i)}.\] Multiplying both sides of this equation by $x^n$ and summing over $n\geq 0$, we obtain 
\begin{align}
B_{k}^{(m+1)}(x)&=\sum_{i=0}^{m+1}(2x^k)^i\binom{m+1}{i}B_{k-1}^{(m+1-i)}(x)\nonumber\\
&=\sum_{i=0}^{m+1}(2x^k)^i\binom{m+1}{i}\left(1+2\sum_{j=1}^{k-1}x^j\right)^{m+1-i}.\label{2gw}
\end{align}
By the binomial theorem, the right-hand side of \eqref{2gw} is equal to
\[\left(1+2\sum_{j=1}^{k-1}x^{j}+2x^{k}\right)^{m+1}=
\left(1+2\sum_{j=1}^{k}x^j\right)^{m+1},\] completing the proof. 
\end{proof}

\begin{corollary}
We have
\begin{align}
&a_{n,k}^{(3)}=\nonumber\\
&\frac{1}{3}\begin{cases}
4n^{3}+6n^{2}+8n+3,&\textnormal{if $0\leq n< k$};\\
12k^{3}-36k^{2}n+36kn^{2}-8n^{3}+6n^{2}+6k+2n+3,&\textnormal{if $k\leq n< 2k$};\\
-84k^{3}+108k^{2}n-36kn^{2}+4n^{3}-72k^{2}+72nk-12n^{2}-6k+8n+3,&\textnormal{if $2k\leq n< 3k$};\\
24k^3 + 36k^2 + 18k + 3,&\textnormal{otherwise}.
\end{cases}\nonumber
\end{align}
\end{corollary}

\begin{proof}
Fixing $k$ and $m$, we denote by $A_{k}^{(m)}(x)$ the generating function for the numbers $a_{n,k}^{(m)}$. Clearly, $A_{k}^{(m)}(x) = B_{k}^{(m)}(x)/(1-x)$. Thus
\begin{align}
A_{k}^{(3)}(x)&=\frac{\left(1+2\sum_{i=1}^{k}x^{i}\right)^{3}}{1-x}\nonumber\\&=\frac{1}{1-x}\sum_{j=0}^{3}\binom{3}{j}\left(2\sum_{i=1}^{k}x^{i}\right)^{j}\nonumber\\&=\sum_{j=0}^{3}\binom{3}{j}(2x)^{j}\frac{(1-x^{k})^{j}}{(1-x)^{j+1}} \nonumber\\&=\frac{-8x^{3k+3}+12x^{2k+3}+12x^{2k+2}-6x^{k+3}-12x^{k+2}+x^{3}-6x^{k+1}+3x^{2}+3x+1}{(1-x)^{4}}\nonumber\\
&=
\sum_{n\geq 0}\binom{n+3}{n}\Bigg(
-8x^{n+3k+3}+12x^{n+2k+3}+12x^{n+2k+2}-6x^{n+k+3}\nonumber\\&\hspace{120 pt}-12x^{n+k+2}+x^{n+3}-6x^{n+k+1}+3x^{n+2}+3x^{n+1}+x^n\Bigg).\nonumber
\end{align}
It follows that, for $0\leq n< k$, we have 
\[a_{n,k}^{(3)} =\binom{n}{3}+3\binom{n+1}{3}+3\binom{n+2}{3}+\binom{n+3}{3}=\frac{4n^{3}+6n^{2}+8n+3}{3}.\]
For $k\leq n< 2k$, we have 
\begin{align}
a_{n,k}^{(3)} &=-6\binom{n-k}{3}-12\binom{n-k+1}{3}+\binom{n}{3}-6\binom{n-k+2}{3}\nonumber\\&
\hspace{169pt}+3\binom{n+1}{3}+3\binom{n+2}{3}+\binom{n+3}{3}\nonumber\\&=\frac{12k^{3}-36k^{2}n+36kn^{2}-8n^{3}+6n^{2}+6k+2n+3}{3}\nonumber.
\end{align}
For $2k\leq n< 3k$, we have 
\begin{align}
a_{n,k}^{(3)} &=12\binom{n-2k}{3}+12\binom{n-2k+1}{3}-6\binom{n-k}{3}-12\binom{n-k+1}{3}\nonumber\\&\hspace{79pt}+\binom{n}{3}-6\binom{n-k+2}{3}+3\binom{n+1}{3}+3\binom{n+2}{3}+\binom{n+3}{3}\nonumber\\&=\frac{-84k^{3}+108k^{2}n-36kn^{2}+4n^{3}-72k^{2}+72nk-12n^{2}-6k+8n+3}{3}\nonumber.
\end{align}
Finally, for $n\geq 3k$, we have 
\begin{align}
a_{n,k}^{(3)} &=-8\binom{n-3k}{3}+12\binom{n-2k}{3}+12\binom{n-2k+1}{3}-6\binom{n-k}{3}\nonumber\\&-12\binom{n-k+1}{3}+\binom{n}{3}-6\binom{n-k+2}{3}+3\binom{n+1}{3}+3\binom{n+2}{3}+\binom{n+3}{3}\nonumber\\&=8k^3 + 12k^2 + 6k + 1\nonumber.\qedhere
\end{align}
\end{proof}

\subsection{Maximum value of a cyclic convolution}

In the next theorem we validate a conjecture stated in \seqnum{A294172} regarding a closed-form formula for the maximal value of the cyclic convolution of the numbers $1,2,\ldots,n$ with themselves.

\begin{theorem}
For an integer $n\geq 1$ define
\[
a_n =\max\left\{\sum_{i=1}^{n}(n-i+1)(1+(i+k\textnormal{ \ensuremath{(}mod }n)))\;:\;1\leq k\leq n\right\}.\]
Then \[a_n = \frac{1}{24}
\begin{cases}
7n^{3}+12n^{2}+8n,&\textnormal{if $n$ is even}; \\7n^{3}+12n^{2}+5n,   &\textnormal{otherwise}. 
\end{cases}\] 
\end{theorem}

\begin{proof}
Let $1\leq k\leq n$. We have
\begin{align}
&\sum_{i=1}^{n}(n-i+1)(1+(i+k\text{ \ensuremath{(}mod }n)))\nonumber\\&=\sum_{i=1}^{n-1-k}(n-i+1)(1+i+k)+\sum_{i=n-k}^{n}(n-i+1)(1+i+k-n)\nonumber\\&=-\frac{n}{2}k^{2}+\left(\frac{n^{2}}{2}-n\right)k+\frac{n^{3}}{6}+n^{2}-\frac{n}{6}.\nonumber
\end{align}
The quadratic function \[f(x)=-\frac{n}{2}x^{2}+\left(\frac{n^{2}}{2}-n\right)x+\frac{n^{3}}{6}+n^{2}-\frac{n}{6}\] obtains its maximum at $x=n/2-1$. From this the assertion follows easily.
\end{proof}

\subsection{The real part of a recursively defined complex sequence}

Let $i=\sqrt{-1}$. The result of the next theorem  validates all the conjectures stated in \seqnum{A309878}, which is concerned with the real part of a recursively defined complex sequence.

\begin{theorem}
Let 
$(b_n)_{n\geq 0}$ be the sequence defined recursively as follows: $b_0=0$ and, for $n\geq 1$,
\begin{equation}\label{11z}
b_n = (n+b_{n-1})(1+i).
\end{equation} 
Let $(a_n)_{n\geq 0}$ be the sequence corresponding to the real part of $(b_n)_{n\geq 0}$, i.e., $a_n=\textnormal{Re}(b_n)$. Then
\[a_n = 2^{\frac{n}{2}+1}\sin\left(\frac{n\pi}{4}\right)-n.\]
\end{theorem}

\begin{proof}
Denote by $B(x)$ the generating function for the sequence $(b_n)_{n\geq 0}$. Multiplying \eqref{11z} by $x^n$ and summing over $n\geq 1$, we obtain
\[
\sum_{n\geq 1}b_{n}x^{n}=\sum_{n\geq1}(n+b_{n-1})(1+i)x^{n}.
\]
Since $b_0=0$, we have
\[
\sum_{n\geq0}b_{n}x^{n}=(1+i)x\sum_{n\geq1}nx^{n-1}+(1+i)x\sum_{n\geq0}b_{n}x^{n}.\]
Thus, \[B(x)=\frac{(1+i)x}{(1-x)^2}+(1+i)xB(x).\] Solving for $B(x)$, we obtain
\begin{align}
B(x)&=\frac{(1+i)x}{(1-x)^{2}(1-(1+i)x)}\nonumber\\&=\frac{-1+i}{(1-x)^{2}}-\frac{2+2i}{1-(2x+i)}+\frac{1+i}{1-x}    \nonumber\\&=\frac{-1+i}{(1-x)^{2}}-\frac{2+2i}{1-i}\frac{1}{1-\frac{2x}{1-i}}+\frac{1+i}{1-x}\nonumber\\&=(-1+i)\sum_{n\geq0}nx^{n}+2i\sum_{n\geq0}x^{n}-\frac{2+2i}{1-i}\sum_{n\geq0}\frac{2^{n}}{(1-i)^{n}}x^{n}\nonumber\\&=\sum_{n\geq0}\left((n+2)i-n-\frac{2^{n+2}}{(1-i)^{n+2}}\right)x^{n}\nonumber\\
&=\sum_{n\geq0}\left(2^{\frac{n}{2}+1}\sin\left(\frac{n\pi}{4}\right)-n+\left((n+2)-2^{\frac{n}{2}+1}\sin\left(\frac{(n+2)\pi}{4}\right)\right)i\right)x^{n},\nonumber
\end{align}
from which the assertion immediately follows.
\end{proof}

\subsection{The generating functions of two sequences}

The results of the next theorem validate two conjectures stated in \seqnum{A294139} and \seqnum{A307684}.

\begin{theorem}
Let $(a_n)_{n\geq 1}$ and $(b_n)_{n\geq 1}$ be the sequences defined by
\begin{align}
a_n &=\sum_{k=1}^{\left\lfloor \frac{n-1}{2}\right\rfloor } \left(2k^2 +2(n-k)^2 +k(n-k)\right),\nonumber\\
b_n& =\sum_{k=1}^{\left\lfloor \frac{n}{3}\right\rfloor }\sum_{i=k}^{\left\lfloor \frac{n-k}{2}\right\rfloor} ik(n-i-k).\nonumber
\end{align} Let $A(x)$ and $B(x)$ be the generating functions of the two sequences, respectively. Then
\begin{align}
A(x)& =\frac{x^{3}(2x^{3}+11x^{2}+11x+12)}{(1-x)^{4}(1+x)^{3}},\nonumber\\
B(x)& = \frac{x^{3}(6x^{8}+14x^{7}+18x^{6}+21x^{5}+23x^{4}+15x^{3}+7x^{2}+3x+1)}{(1-x)^{6}(1+x)^{3}(x^{2}+x+1)^{4}}.\nonumber
\end{align}
\end{theorem}

\begin{proof}
Write $n=2s$ if $n$ is even and $n=2s+1$ if $n$ is odd, for a nonnegative integer $s$. Using Faulhaber's formula (e.g., \cite[p.\ 106]{CG}), we have
\begin{align}
a_n &=\sum_{k=1}^{\left\lfloor \frac{n-1}{2}\right\rfloor } \left(2k^2 +2(n-k)^2 +k(n-k)\right)\nonumber\\
&=\frac{\left(\left\lfloor \frac{n-1}{2}\right\rfloor -1\right)\left\lfloor \frac{n-1}{2}\right\rfloor \left(2\left\lfloor \frac{n-1}{2}\right\rfloor -1\right)}{2}-\frac{3n\left(\left\lfloor \frac{n-1}{2}\right\rfloor -1\right)\left\lfloor \frac{n-1}{2}\right\rfloor }{2}+2n^{2}\left(\left\lfloor \frac{n-1}{2}\right\rfloor -1\right)\nonumber\\
&=\begin{cases}
6s^{3}-\frac{13}{2}s^{2}+\frac{1}{2}s, & \textnormal{if $n$ is even};\\
6s^{3}+5s^{2}+s, & \textnormal{otherwise}.
\end{cases}\nonumber
\end{align}
The generating functions for the sequences $\left(6s^{3}-\frac{13}{2}s^{2}+\frac{1}{2}s \right)_{s\geq0}$ and $\left(6s^{3}+5s^{2}+s \right)_{s\geq0}$  
are
\begin{align}
A_1(x):=&\frac{6x(x^{2}+4x+1)}{(1-x)^{4}}-\frac{13x(x+1)}{2(1-x)^{3}}+\frac{x}{2(1-x)^{2}},\nonumber\\
A_2(x):=&\frac{6x(x^{2}+4x+1)}{(1-x)^{4}}+\frac{5x(x+1)}{(1-x)^{3}}+\frac{x}{(1-x)^{2}},\nonumber
\end{align} respectively.
The assertion follows from calculating $A(x)=A_1(x^2)+xA_2(x^2)$.

The calculation of $B(x)$ is similar, but more involved. First, write $n=6s+r$, where $s\geq0$ and $0\leq r\leq 5$ are integers. Now,
\begin{align}
&b_n=\nonumber\\
&=\sum_{k=1}^{\left\lfloor \frac{n}{3}\right\rfloor }\sum_{i=k}^{\left\lfloor \frac{n-k}{2}\right\rfloor }ik(n-i-k)\nonumber\\
&=\sum_{k=1}^{\left\lfloor \frac{n}{3}\right\rfloor }\sum_{i=k}^{\left\lfloor \frac{n-k}{2}\right\rfloor }ikn-\sum_{k=1}^{\left\lfloor \frac{n}{3}\right\rfloor }\sum_{i=k}^{\left\lfloor \frac{n-k}{2}\right\rfloor }i^{2}k-\sum_{k=1}^{\left\lfloor \frac{n}{3}\right\rfloor }\sum_{i=k}^{\left\lfloor \frac{n-k}{2}\right\rfloor }ik^{2}  \nonumber\\&=\sum_{k=1}^{\left\lfloor \frac{n}{3}\right\rfloor }\left(\frac{k}{6}\left\lfloor \frac{n-k}{2}\right\rfloor \left(\left\lfloor \frac{n-k}{2}\right\rfloor +1\right)\left(3(n-k)-2\left\lfloor \frac{n-k}{2}\right\rfloor -1\right)+\frac{k^{2}(k-1)(5k-3n-1)}{6}\right)\nonumber\\
&=\begin{cases}
\frac{54}{5}s^{5}+\frac{27}{4}s^{4}+\frac{7}{6}s^{3}-\frac{3}{4}s^{2}+\frac{1}{30}s, & \textnormal{if } n=6s;\\
\frac{54}{5}s^{5}+\frac{63}{4}s^{4}+6s^{3}+\frac{5}{4}s^{2}+\frac{1}{5}s, & \textnormal{if } n=6s+1;\\
\frac{54}{5}s^{5}+\frac{99}{4}s^{4}+\frac{39}{2}s^{3}+\frac{25}{4}s^{2}+\frac{7}{10}s, & \textnormal{if } n=6s+2;\\
\frac{54}{5}s^{5}+\frac{135}{4}s^{4}+\frac{125}{3}s^{3}+\frac{103}{4}s^{2}+\frac{241}{30}s+1, & \textnormal{if } n=6s+3;\\
\frac{54}{5}s^{5}+\frac{171}{4}s^{4}+\frac{129}{2}s^{3}+\frac{185}{4}s^{2}+\frac{157}{10}s+2, & \textnormal{if } n=6s+4;\\
\frac{54}{5}s^{5}+\frac{207}{4}s^{4}+96s^{3}+\frac{349}{4}s^{2}+\frac{196}{5}s+7, & \textnormal{otherwise}.
\end{cases}\nonumber
\end{align}
Proceeding as in the calculation of $A(x)$, the asserted expression for $B(x)$ follows.
\end{proof}
\subsection{Words avoiding the patters \texorpdfstring{$z,z+1,z$}{} and \texorpdfstring{$z,z,z+1$}{}}
Let $k\geq 2$ be an integer and denote by $[k]$ the set $\{1,2,\ldots, k\}$. For a nonnegative integer $n$, a {\it word over $k$ of length $n$} is an element of $[k]^n$. A word $w_1\cdots w_n\in [k]^n$  {\it avoids the pattern $z,z+1,z$} if no $1\leq i\leq n-2$ and $z\in[k]$ exist such that $w_i = z, w_{i+1}=z+1$, and $w_{i+2}=z$. Avoidance of the other patter is defined similarly. The first part of the following theorem corresponds to \seqnum{A005251} by taking $k=2$, to \seqnum{A098182} by taking $k=3$ (giving the sequence a combinatorial interpretation), and to \seqnum{A206790} by taking $k=4$ (proving the conjectures stated there).
The second part corresponds to \seqnum{A000071} by taking $k=2$, to \seqnum{A206727} by taking $k=3$ (proving the conjectures stated there), and to \seqnum{A206570} by taking $k=4$ (proving the conjectures stated there).

\begin{theorem}
\begin{enumerate}
\item Denote by $f_{k}(n)$ the number of words over $k$ of length $n$ that avoid the pattern $z,z+1,z$. Then, for $n\geq 3$, the numbers $f_{k}(n)$ satisfy the recursion 
\[f_{k}(n)=kf_{k}(n-1)-f_{k}(n-2)+f_{k}(n-3),\] with initial values $f_{k}(0)=1,f_{k}(1)=k$, and $f_{k}(2)=k^2$. In particular, the corresponding generating function is given by \[\frac{1+x^{2}}{1-kx+x^{2}-x^{3}}.\]
\item Denote by $f_{k}(n)$ the number of words over $k$ of length $n$ that avoid the pattern $z,z,z+1$. Then, for $n\geq 2k-1$, the numbers $f_{k}(n)$ satisfy the recursion 
\[f_{k}(n)=\sum_{i=0}^{k-1}(-1)^i(k-i)f_{k}(n-2i-1).\]  Furthermore, the corresponding generating function is given by \[\frac{1}{1-\sum_{i=0}^{k-1}(-1)^i(k-i)x^{2i+1}}.\] 
\end{enumerate}    
\end{theorem}

\begin{proof}
\begin{enumerate}
\item For $u,v\in [k]$, denote by $f_{k,u,v}(n)$ the number of words over $k$ which avoid the pattern $z,z+1,z$, whose last two letters are $u,v$. For $n\geq 3$ we have \[f_{k,u,v}(n)=\begin{cases}
\sum_{t=1}^{k}f_{k,t,u}(n-1), & \text{if \ensuremath{v\neq u-1;}}\\
\sum_{t=1}^{k}f_{k,t,u}(n-1)-f_{k,u-1,u}(n-1), & \text{otherwise}.
\end{cases}\] Now, 
\begin{align}
f_{k}(n)&=\sum_{u,v=1}^kf_{k,u,v}(n)\nonumber\\
&=\sum_{u,v=1}^k\sum_{t=1}^kf_{k,t,u}(n-1)-\sum_{u=2}^kf_{k,u-1,u}(n-1)\nonumber\\
&=kf_{k}(n-1)-\sum_{u=2}^kf_{k,u-1,u}(n-1). \label{qq1}
\end{align} Furthermore, 
\begin{align}
\sum_{u=2}^kf_{k,u-1,u}(n-1)&=\sum_{u=2}^k\sum_{t=1}^kf_{k,t,u-1}(n-2)\nonumber\\
&=\sum_{t,u=1}^kf_{k,t,u}(n-2) - \sum_{t=1}^kf_{k,t,k}(n-2)\nonumber\\
&=f_k(n-2)-f_k(n-3).\label{6qq}
\end{align} Substituting \eqref{6qq} in \eqref{qq1}, the assertion follows.
\item For $u,v\in [k]$, denote by $f_{k,u,v}(n)$ the number of words over $k$ which avoid the pattern $z,z,z+1$, whose last two letters are $u,v$. For $n\geq 3$ we have \[f_{k,u,v}(n)=\begin{cases}
\sum_{t=1}^{k}f_{k,t,u}(n-1), & \text{if \ensuremath{v\neq u+1;}}\\
\sum_{t=1}^{k}f_{k,t,u}(n-1)-f_{k,u,u}(n-1), & \text{otherwise}.
\end{cases}\] Now, 
\begin{align}
f_{k}(n)&=\sum_{u,v=1}^kf_{k,u,v}(n)\nonumber\\
&=\sum_{u,v=1}^k\sum_{t=1}^kf_{k,t,u}(n-1)-\sum_{u=1}^{k-1}f_{k,u,u}(n-1)\nonumber\\
&=kf_{k}(n-1)-\sum_{u=1}^{k-1}f_{k,u,u}(n-1). \label{qq10}
\end{align} Furthermore, 
\begin{align}
\sum_{u=1}^{k-1}f_{k,u,u}(n-1)&=\sum_{u=1}^{k-1}\sum_{t=1}^kf_{k,t,u}(n-2)\nonumber\\
&=\sum_{t,u=1}^kf_{k,t,u}(n-2) - \sum_{t=1}^kf_{k,t,k}(n-2)\nonumber\\
&=f_k(n-2)-f_k(n-3)+f_{k,k-1,k-1}(n-3).\label{6qq0}
\end{align} We have 
\[
f_{k,k-1,k-1}(n-3)=\sum_{t=1}^{k}f_{k,t,k-1}(n-4)
=f_{k}(n-5)-f_{k,k-2,k-2}(n-5).\] Proceeding in this manner, we arrive at
\[f_{k,1,1}(n-2k+1)=\sum_{t=1}^{k}f_{k,t,1}(n-2k)=f_{k}(n-2k-1).\] Substituting backwards, we see that \[f_{k}(n)=kf_{k}(n-1)-f_{k}(n-2)+\sum_{i=1}^k(-1)^{i+1}f_{k}(n-2i-1).\] Now, in general, if $a_1,\ldots,a_s$ are real numbers and $g(n)$ is a sequence satisfying the recursion \[g(n)=\sum_{t=1}^{s}a_{t}g(n-t),\] then the corresponding generating function $G(x)$ is given by \[G(x)=\frac{g(0)+\sum_{n=1}^{s-1}\left(g(n)-\sum_{t=1}^{n}a_{t}g(n-t)\right)x^{n}}{1-\sum_{t=1}^{s}a_{t}x^{t}}.\]
It is not hard to see that, in our case,  the nominator evaluates to $1+x^2$. It follows that the generating function for the numbers $f_k(n)$ is \[\frac{1+x^2}{1-kx+x^2-\sum_{i=1}^k(-1)^{i+1}x^{2i+1}}.\] Finally, the polynomial $1-kx+x^2-\sum_{i=1}^k(-1)^{i+1}x^{2i+1}$ factors as 
\[(1+x^2)\left(1-\sum_{i=0}^{k-1}(-1)^i(k-i)x^{2i+1}\right),\] concluding the proof.
\end{enumerate}
\end{proof}

In the following theorem we resolve some of the conjectures stated in \seqnum{A269467}. The sequence is concerned with the number $f_k(n)$ of words over $[k]$ of length $n$ with no repeated letter equal to the previous repeated letter.

\begin{theorem}
Let $F_k(x)$ denote the generating function of the sequence $(f_k(n))_{n\geq 1}$.
Then 
\begin{equation}\label{ttp1}
F_k(x) = -\frac{(k-2)x^{3}+2(k-2)x^{2}+(k-3)x-1}{2(k-1)^{2}x^{3}+(k-1)(k-4)x^{2}+(3-2k)x+1}.
\end{equation}
\end{theorem}

\begin{proof}
For $0\leq u\leq k$ and $1\leq v\leq k$ denote by $f_{k,u,v}(n)$ the number of words over $[k]$ of length $n$ whose last letter is $v$ and whose last repeated letter is $u$. Notice that $u=0$ encodes words having no repeated letters and therefore $\sum_{v=1}^kf_{k,0,v}(n)=k(k-1)^{n-1}$. For $n\geq 3$ we have \[f_{k,u,v}(n)=\begin{cases}
\sum_{1\leq t\leq k,t\neq v}f_{k,u,t}(n-1)+\sum_{0\leq t\leq k,t\neq v}f_{k,t,u}(n-1), & \text{if \ensuremath{v=u;}}\\
\sum_{1\leq t\leq k,t\neq v}f_{k,u,t}(n-1), & \text{if \ensuremath{v\neq u.}}
\end{cases}\] Then 
\begin{align}
f_{k}(n)& =\sum_{u=0}^{k}\sum_{v=1}^{k}f_{k,u,v}(n)    \nonumber\\
&=\sum_{u=0}^{k}\sum_{v=1}^{k}\sum_{1\leq t\leq k,t\neq v}f_{k,u,t}(n-1)+\sum_{v=1}^{k}\sum_{0\leq t\leq k,t\neq v}f_{k,t,v}(n-1)\nonumber\\
&=\sum_{v=1}^{k}\sum_{u=0}^{k}\left(\sum_{t=1}^{k}f_{k,u,t}(n-1)-f_{k,u,v}(n-1)\right)+\sum_{v=1}^{k}\left(\sum_{t=0}^{k}f_{k,t,v}(n-1)-f_{k,v,v}(n-1)\right)\nonumber\\
&=kf_{k}(n-1)-\sum_{v=1}^{k}f_{k,v,v}(n-1).\label{it2}
\end{align}
Now, 
\begin{align}
\sum_{v=1}^{k}f_{k,v,v}(n-1)&=\sum_{v=1}^{k}\left(\sum_{t=1}^{k}f_{k,v,t}(n-2)+\sum_{t=0}^{k}f_{k,t,v}(n-2)-2f_{k,v,v}(n-2)\right)   \nonumber\\
&=2f_{k}(n-2)-\sum_{t=1}^{k}f_{k,0,t}(n-2)-2\sum_{v=1}^{k}f_{k,v,v}(n-2)\nonumber\\
&=2f_{k}(n-2)-k(k-1)^{n-3}-2\sum_{v=1}^{k}f_{k,v,v}(n-2).\label{it1}
\end{align}
Iterating \eqref{it1} until we reach $\sum_{v=1}^{k}f_{k,v,v}(2)$, which is equal to $k$, and substituting the result into \eqref{it2}, we obtain the equation \[f_{k}(n)=kf_{k}(n-1)+\sum_{i=1}^{n-3}\left((-2)^{i}f_{k}(n-1-i)+(-2)^{i-1}k(k-1)^{n-2-i}\right)-(-2)^{n-3}k.\] Multiplying both sides of the equation by $x^n$ and summing over $n\geq 4$, we may (eventually) solve for $F_k(x)$ and obtain \eqref{ttp1}.
\end{proof}

\begin{corollary}
The denominator of $F_k(x)$ confirms the conjectured recurrences stated in \seqnum{A269467}. 
\end{corollary}

\subsection{A greedily defined integer sequence}

In the following theorem we show how sequence \seqnum{A128135} emerges as a subsequence of a sequence defined by a greedy integer recurrence. Such greedily defined integer sequences have been studied by Venkatachala \cite{V}, Avdispahi\'{c} and Zejnulahi \cite{AZ}, and Shallit \cite{S}. While the result we obtain does not directly settle a conjecture from the OEIS, we noticed it while we were working on a conjecture stated in \seqnum{A248982}. We shall elaborate on it at the end of this section.

\begin{theorem}
Let $a_1=1$ and, for $n\geq 2$, let $a_n$ be the least positive integer such that the average of $a_1,\ldots,a_{n-1}$ is a power of $2$. Then
\[a_n = 
\begin{cases}
(n+1)2^{\frac{n}{2}-1},&\textnormal{if } n \textnormal{ is even};\\   
2^{\frac{n-1}{2}}, & \textnormal{otherwise}.    \end{cases}
\]
\end{theorem}

\begin{proof}
Set $s_n=\sum_{i=1}^n a_i$. We claim that $s_n=n2^{\left\lfloor\frac{n}{2}\right\rfloor}$ (\seqnum{A132344}). To see that, We proceed by induction on $n$. The base case, namely $n=1$, obviously holds. Now, assume that the assertion holds for every $1\leq i\leq n-1$, where $n\geq 2$. Let $\ell$ be a nonnegative integer such that $s_n=n2^\ell$. We have 
\begin{equation}\label{1k}
a_n = s_n - s_{n-1} = n2^\ell- (n-1)2^{\left\lfloor\frac{n-1}{2}\right\rfloor}.
\end{equation} Clearly, if $\ell\geq \lfloor(n-1)/2\rfloor$, then $a_n > 0$. We claim that the converse also holds. Indeed, suppose that $a_n >0$ but $\ell<\lfloor(n-1)/2\rfloor$. Then \[n2^{\left\lfloor \frac{n-1}{2}\right\rfloor -1}-(n-1)2^{\left\lfloor \frac{n-1}{2}\right\rfloor }>0\iff\frac{n}{n-1}>2.\] But since $n\geq 2$, we have $n/(n-1) \leq 2$. 

Now, assume that $n$ is even. If $\ell=\lfloor(n-1)/2\rfloor$, then, by \eqref{1k} and the induction hypothesis, $a_n = 2^{\left\lfloor\frac{n-1}{2}\right\rfloor} = 2^{\frac{n-2}{2}}= a_{n-1}$, in violation of the distinctness condition. Trying the next best candidate $\ell=\lfloor(n-1)/2\rfloor+1$, we have, by \eqref{1}, \[a_n = n2^{\left\lfloor\frac{n-1}{2}\right\rfloor+1}- (n-1)2^{\left\lfloor\frac{n-1}{2}\right\rfloor}=(n+1)2^{\left\lfloor\frac{n-1}{2}\right\rfloor},\] which is obviously adequate. Thus, $s_n = n2^{\left\lfloor\frac{n}{2}\right\rfloor}$.

Consider $a_{n+1}$ now and let $\ell$ be a nonnegative integer such that $s_{n+1}=(n+1)2^\ell$. We have
\[a_{n+1} = s_{n+1} - s_n = (n+1)2^\ell - n2^{\left\lfloor\frac{n}{2}\right\rfloor}.\] Here, $\ell=\lfloor n/2 \rfloor$ is possible, leading to $a_{n+1}=2^{\left\lfloor\frac{n}{2}\right\rfloor}$ and $s_{n+1} = (n+1)2^{\left\lfloor\frac{n}{2}\right\rfloor}$, concluding the proof of the induction step.
\end{proof}

As mentioned earlier, we noticed the statement of the previous theorem while we were working on a conjecture stated in \seqnum{A248982}, which is defined to be the sequence of distinct least positive numbers such that the average of the first $n$ terms is a Fibonacci number. Let $(a_n)_{n\geq1}$ be this sequence. Refining the conjecture stated in \seqnum{A248982} regarding a closed-form formula for $(a_n)_{n\geq1}$, it seems that, for $n\geq 10$, we have
\[a_n = 
\begin{cases}
nF\left(\frac{n}{2}+3\right)-(n-1)F\left(\frac{n}{2}+2\right),&\textnormal{if $n$ is even};\\    
F\left(\frac{n+1}{2}+2\right), & \textnormal{otherwise}.\\    
\end{cases}
\]
The proof of this should go along the same lines as the proof of the previous theorem. Nevertheless, we were not able to show that the intersection of the sets 
\begin{align}
&\left\{nF\left(\frac{n}{2}+3\right)-(n-1)F\left(\frac{n}{2}+2\right)\;:\;\text{$n\geq 1$ is even} \right\},\nonumber\\
&\left\{F\left(\frac{n+1}{2}+2\right)\;:\;\text{$n\geq 1$ is odd} \right\},\nonumber
\end{align}
is empty.   

\begin{thebibliography}{99}

\bibitem{AZ}
M.~Avdispahi\'{c} and F.~Zejnulahi, An integer sequence with a divisibility property, {\it Fibonacci Quart.} {\bf 58} (2020), 321--333.

\bibitem{B}
B.~Bajnok, {\it Additive Combinatorics: A Menu of Research Problems}, Chapman and Hall/CRC, 2018. 

\bibitem{BW}
R.~P.~Boas and J.~W.~Wrench, Partial sums of the harmonic series, {\it Amer. Math. Monthly} {\bf 78} (1971), 864--870.

\bibitem{CG}
J.~H.~Conway and R.~Guy, {\it The Book of Numbers}, Springer Science \& Business Media, 1998.

\bibitem{K}
K.~Knopp, {\it Theory and Application of Infinite Series}, Courier Corporation, 1990.

\bibitem{H}
D.~C.~Hoaglin, F.~Mosteller, and J.~W.~Tukey, {\it Understanding Robust and Exploratory Data Analysis}, John Wiley \& Sons, 2000.

\bibitem{S}
J.~Shallit, Proving properties of some greedily-defined integer recurrences via automata theory, {\it Theoretical Computer Science} {\bf 988} (2024).

\bibitem{SL}
N.~J.~A.~Sloane, The On-Line   Encyclopedia of Integer   Sequences, OEIS  Foundation  Inc., \url{https://oeis.org}.

\bibitem{V}
B.~J.~Venkatachala, A curious bijection on natural numbers, {\it J. Integer Sequences} {\bf 12}
(2009), Article 09.8.1.

\bibitem{W}
D.~V.~Widder, {\it Advanced Calculus}, Prentice-Hall, 1947.

\end{thebibliography}
\end{document}